\documentclass[11pt,oneside]{article}

\usepackage[utf8]{inputenc}
\usepackage[english]{babel}
\usepackage{amsmath}
\usepackage{amsfonts}
\usepackage{amsthm}
\usepackage{amssymb}
\usepackage{color}
\usepackage{enumitem}
\usepackage{bbm}
\usepackage{hyperref}
\usepackage[a4paper]{geometry}
\usepackage{setspace}

\theoremstyle{definition}
\newtheorem{definition}{Definition}[section]

\theoremstyle{plain}
\newtheorem{theorem}[definition]{Theorem}
\newtheorem{proposition}[definition]{Proposition}
\newtheorem{corollary}[definition]{Corollary}
\newtheorem{lemma}[definition]{Lemma}

\theoremstyle{remark}
\newtheorem{remark}[definition]{Remark}

\newcommand{\bb}[1]{\boldsymbol{#1}}
\newcommand{\Z}{\boldsymbol{Z}}
\newcommand{\x}{\boldsymbol{x}}
\newcommand{\y}{\boldsymbol{y}}
\newcommand{\h}{\boldsymbol{h}}
\newcommand{\ttt}{\boldsymbol{t}}
\newcommand{\uu}{\boldsymbol{u}}
\newcommand{\NN}{\mathbb{N}}

\newcommand{\RR}{\ensuremath{\mathbb{R}}}
\newcommand{\EE}{\ensuremath{\mathbb{E}}}
\newcommand{\Var}{\ensuremath{\text{\normalfont Var}}}
\newcommand{\Cov}{\ensuremath{\text{\normalfont Cov}}}
\newcommand{\T}{\tilde}
\newcommand{\G}{{\boldsymbol \gamma}}
\newcommand{\TG}{{\boldsymbol {\tilde\gamma}}}
\newcommand{\0}{\boldsymbol{0}}

\allowdisplaybreaks
\title{Characterization Theorems for Pseudo-Variograms}

\author{Christopher D\"orr\footnote{Institute of Mathematics, University of Mannheim, 68159 Mannheim, Germany. Email: chrdoerr@mail.uni-mannheim.de} {} and Martin Schlather\footnote{Institute of Mathematics, University of Mannheim, 68159 Mannheim, Germany.}}
\date{\today}

\begin{document}

\maketitle

\setstretch{1.25}
\begin{abstract}
	Pseudo-variograms appear naturally in the context of multivariate Brown-Resnick processes, and are a useful tool for analysis and prediction of multivariate random fields. We give a necessary and sufficient criterion for a matrix-valued function to be a pseudo-variogram, and further provide a Schoenberg-type result connecting pseudo-variograms and multivariate correlation functions. By means of these characterizations, we provide extensions of the popular univariate space-time covariance model of Gneiting to the multivariate case. 
\end{abstract}

{\small
	}{
	\phantom{\textit{2010 MSC}:}{Secondary 60G10} 
}

\section{Introduction} 

With increasing availability of multivariate data and considerable improvements of computational feasibility, multivariate random fields have become a significant part of geostatistical modeling throughout recent years.

These random fields are usually assumed to be either second-order stationary or intrinsically stationary. An $m$-variate random field $\{\Z(\x)=(Z_1(\x),\dots,Z_m(\x))^\top, \x\in\RR^d \}$ is second-order stationary if it has a constant mean and if its auto- and cross-covariances $\Cov(Z_i(\x+\h), Z_j(\x))$, $ \x, \h \in \RR^d$, $i,j=1,\dots,m$, exist and are functions of the lag $\h$ only. It is called intrinsically stationary, if the increment process \\$\{\Z(\x+\h)-\Z(\x), \x \in \RR^d\}$ is second-order stationary for all $\h \in \RR^d$. In this case, the function $\TG: \RR^d \rightarrow \RR^{m \times m}$,  \begin{equation*}
\T\gamma_{ij}(\h)= \frac12\Cov(Z_i(\x+\h)-Z_i(\x), Z_j(\x+\h)-Z_j(\x)), \quad \x,\h \in \RR^d, \quad i,j=1,\dots,m,
\end{equation*} is well-defined and is called cross-variogram \cite{myers1982matrix}.
If we additionally assume that \\ $Z_i(\x+\h)-Z_j(\x)$ is square integrable and that $\Var(Z_i(\x + \h) - Z_j(\x))$ does not depend on $\x$ for all $\x, \h\in \RR^d$, $i,j=1,\dots,m$, then we can also define the so-called pseudo-variogram  $\G: \RR^d \rightarrow \RR^{m \times m}$ \cite{myers1991pseudo} via
\begin{equation*}
\gamma_{ij}(\h)= \frac12\Var(Z_i(\x+\h)-Z_j(\x)), \quad \x,\h \in \RR^d, \quad i,j=1,\dots,m.
\end{equation*}
Obviously, the diagonal entries of pseudo- and cross-variograms coincide and contain univariate variograms $\gamma_{ii}(\h)= \frac12\Var(Z_i(\x+\h)-Z_i(\x))$, $i=1,\dots,m$ \cite{matheron1972leccon,matheron1973intrinsic}.

Both cross- and pseudo-variograms are commonly used in geostatistics to capture the degree of spatial dependence \cite{chen2019parametric}. There is some controversy on which one to use, since both have their benefits and drawbacks. The cross-variogram, on the one hand, is well-defined under weaker assumptions and might be considered as the natural extension of the concept of a variogram from the univariate to the multivariate case \cite{chen2019parametric}. However, it requires measurements of the quantities of interest at the same locations for estimation in practical applications \cite{chen2019parametric}. Moreover, it only reproduces the symmetric part of a cross-covariance function of a stationary random field, see \cite{wackernagel2003multivariate}, for instance. The pseudo-variogram, on the other hand, can capture asymmetry, and provides optimal co-kriging predictors without imposing any symmetry assumption on the cross-dependence structure \cite{ver1993multivariable}, but is difficult to interpret in practice due to considering differences of generally different physical quantities, cf.\,\cite{cressie1998variance} and their account on it.

From a theoretical perspective, pseudo-variograms are interesting objects, since 
they are not only found in multivariate geostatistics, but also appear naturally in extreme value theory in the context of multivariate Brown-Resnick processes \cite{genton2015maxstable,oesting2017}. However, pseudo-variograms, in contrast to cross-variograms, have not been well enough understood yet. So far, elementary properties \cite{du2012}, their relation to cross-variograms and cross-covariance functions \cite{myers1991pseudo, papritz1993}, their applicability to co-kriging \cite{myers1991pseudo, ver1993multivariable, cressie1998variance}, and limiting behaviour \cite{papritz1993,oesting2017} are known, but a concise necessary and sufficient criterion for a matrix-valued function to be a pseudo-variogram is missing. This lack of an equivalent characterization makes it very difficult to show the validity of a function as a pseudo-variogram, cf.\,\cite[p.\,239]{ma2011class}, for instance, unless it can be led back to an explicit construction of a random field as in \cite{chen2019parametric} or \cite{oesting2017}. 

Equivalent characterizations are well-known for univariate variograms, see \cite{matheron1972leccon,gneiting2001}, for instance, and involve the notion of conditional negative definiteness. These characteristics are intimately connected with a result which can be mainly attributed to Schoenberg \cite{berg1984harmonic, schoenberg1938metric}, implying that a function $\gamma: \RR^d \rightarrow \RR$ is a univariate variogram if and only if $\exp(-t\gamma)$ is a (univariate) correlation function for all $t >0$ \cite{gneiting2001}. Such a characterization of $\gamma$ in the multivariate case, however, is unknown in geostatistical literature. For cross-variograms, there is a result for the if part \cite[Theorem 10]{ma2011vector}. The only if part is false in general, see \cite[Remark 2]{schlather2010}, for instance.

The aim of this article is to fill these gaps. The key ingredient to do this is to apply a stronger notion of conditional negative definiteness for matrix-valued functions than the predominant one in geostatistical literature. We 
discuss this notion in Section~\ref{sec:negativedefinite}, and provide a first characterization of pseudo-variograms in these terms. This characterization leads to a Schoenberg-type result in terms of pseudo-variograms in Section~\ref{sec:schoenberg}, thus making a case for proponents of pseudo-variograms, at least from a theoretical standpoint. In Section~\ref{sec:gneitingmultivariat}, we apply this characterization and illustrate its power by extending versions of the very popular space-time covariance model of Gneiting \cite{gneiting2002} to the multivariate case.

\section{Conditional negative definiteness for matrix-valued functions}\label{sec:negativedefinite}
Real-valued conditionally negative definite functions are essential to characterize variograms. A function $\gamma: \RR^d \rightarrow \RR$ is a variogram, if and only if $\gamma(0)=0$ and $\gamma$ is conditionally negative definite, i.e.\,$\gamma$ is symmetric and for all $n \ge 2$, $\bb{x_1},\dots,\bb{x_n} \in \RR^d$, $a_1,\dots,a_n \in \RR$ such that $\sum_{k=1}^n a_k=0$, the inequality $\sum_{i=1}^n\sum_{j=1}^n a_i \gamma(\bb{x_i}-\bb{x_j})a_j \le 0$ holds  \cite{matheron1972leccon}. An extended notion of conditional negative definiteness for matrix-valued functions is part of a characterization of cross-variograms. A function $\TG: \RR^d \rightarrow \RR^{m \times m}$ is a cross-variogram, if and only if $\TG(\0) =\0$, $\TG(\h)=\TG(-\h)=\TG(\h)^\top$ and \begin{equation} \label{eq:weaknd} \sum_{i=1}^n\sum_{j=1}^n \bb{a_i^\top} \TG(\bb{x_i}-\bb{x_j})\bb{a_j} \le 0 \end{equation} for $n \ge 2$, $\bb{x_1},\dots,\bb{x_n} \in \RR^d$, $\bb{a_1},\dots,\bb{a_n} \in \RR^m$ such that $\sum_{k=1}^n \bb{a_k} =\0$ \cite{ma2011vector}.
A function satisfying condition \eqref{eq:weaknd} is called an almost negative-definite matrix-valued function in \cite[p.\,40]{xie1994positive}. 

A pseudo-variogram $\G$ has similar, but only necessary properties, see \cite{du2012}. It holds that $\gamma_{ii}(\0)=0$, and $\gamma_{ij}(\h)=\gamma_{ji}(-\h)$, $i,j=1,\dots,m$. Additionally, a pseudo-variogram is an almost negative-definite matrix-valued function as well, but inequality~\eqref{eq:weaknd}, loosely speaking, cannot enforce non-negativity on the secondary diagonals. 
Therefore, we consider the following stronger notion of conditional negative definiteness, see \cite{gesztesy2016conditional}.

\begin{definition} \label{def:strongnd}
	A function $\G: \RR^d \rightarrow \RR^{m \times m}$ is called conditionally negative definite, if 
	\begin{subequations}{
			\begin{align}
			\gamma_{ij}(\h)=\gamma_{ji}(-\h), \quad i,j=1,\dots,m, \label{eq:hermitian} \\
			\sum_{i=1}^n\sum_{j=1}^n \bb{a_i^\top} \TG(\bb{x_i}-\bb{x_j})\bb{a_j} \le 0, \label{eq:strongnd}
			\end{align}
	}\end{subequations}
	for all $n \in \NN$, $\bb{x_1},\dots,\bb{x_n} \in \RR^d$, $\bb{a_1},\dots,\bb{a_n} \in \RR^m$ such that $\bb{1_m^\top}\sum_{k=1}^n \bb{a_k} =0$ with $\bb{1_m}:=~(1,\dots, 1)^\top \in \RR^m$.
\end{definition}

Obviously, the set of conditionally negative definite matrix-valued functions is a convex cone which is closed under integration and pointwise limits, if existing. In the univariate case, the concepts of conditionally and almost negative definite functions coincide, reproducing the traditional notion of real-valued conditionally negative definite functions.
The main difference between them is the broader spectrum of vectors for which inequality \eqref{eq:strongnd} has to hold, in that the sum of all components has to be zero instead of each component of the sum itself. This modification particularly includes sets of linearly independent vectors in the pool of admissible test vector families, resulting in more restrictive conditions on the secondary diagonals. Indeed, choosing $n=2$, $\bb{x_1}=\h \in \RR^d$, $\bb{x_2}=\0$, and $\bb{a_1}=\bb{e_i}$, $\bb{a_2}=-\bb{e_j}$ in Definition \ref{def:strongnd} with $\{\bb{e_1},\dots,\bb{e_m}\}$ denoting the canonical basis in $\RR^m$, we have $\gamma_{ij}(\h) \ge 0$ for a conditionally negative definite function $\G: \RR^d \rightarrow \RR^{m \times m}$ with $\gamma_{ii}(\0) = 0$, $i,j=1,\dots,m$, fitting the non-negativity of a pseudo-variogram. In fact, the latter condition on the main diagonal and the conditional negative definiteness property are sufficient to characterize pseudo-variograms.  

\begin{theorem}\label{thm:characterization1}
	Let $\G: \RR^d \rightarrow \RR^{m \times m}$. Then there exists a centred Gaussian random field $\Z$ on $\RR^d$ with pseudo-variogram $\G$, if and only if
	$\gamma_{ii}(\0)=0$, $i=1,\dots,m$, and $\G$ is conditionally negative definite.
\end{theorem}

\begin{proof} 
	The proof is analogous to the univariate one in \cite{matheron1972leccon}. Let $\Z$ be an $m$-variate random field with pseudo-variogram $\G$. Obviously, $\gamma_{ii}(\0) = 0$ and $\gamma_{ij}(\h)=\gamma_{ji}(-\h)$ for all $\h \in \RR^d$, $i,j=1,\ldots, m$. Define an $m$-variate random field $\bb{\tilde Z}$ via $\tilde Z_i(\x)=Z_i(\x)-Z_1(\0), \x \in \RR^d, i=1,\dots,m$. Then $\Z$ and $\bb{\tilde Z}$ have the same pseudo-variogram, and  
	$$
	\Cov(\tilde Z_i(\x),\tilde Z_j(\y)) = \gamma_{i1}(x)+\gamma_{j1}(y)-\gamma_{ij}(x-y), $$ cf. also \cite[Equation (6)]{papritz1993},
	i.e.\,$$
	\Cov(\bb{\tilde Z}(\x), \bb{\tilde Z}(\y)) = \bb{\gamma_1}(\x)\bb{1_m^\top} + \bb{1_m
		\gamma_1^\top}(\y) - \gamma(\x - \y), \quad \x, \y \in \RR^d, 
	$$
	with $\bb{\gamma_1}(\x):=(\gamma_{11}(\x),\dots,\gamma_{m1}(\x))^\top$.
	For $\bb{1_m^\top} \sum_{k=1}^m \bb{a_k}=0$, we thus have
	\begin{eqnarray*}
		0 & \le  &    \Var\left(\sum_{i=1}^n \bb{a_i^\top}\bb{\tilde Z}(\bb{x_i})\right)
		=
		\sum_{i=1}^n \sum_{j=1}^n \bb{a_i^\top} \left(\bb{\gamma_1}(\bb{x_i})
		\bb{1_m^\top} +
		\bb{1_m}
		\bb{\gamma_1^\top}(\bb{x_j})- \G(\bb{x_i} -\bb{x_j})\right) \bb{a_j}
		\\&  =& -\sum_{i=1}^n \sum_{j=1}^n \bb{a_i^\top}  \G(\bb{x_i} -\bb{x_j}) \bb{a_j}.
	\end{eqnarray*}
	Now let $\G$ be conditionally negative definite and $\gamma_{ii}(\0)=0$, $i=1,\dots,m$. \\Let  $\bb{a_1},\dots, \bb{a_n} \in \RR^m, \bb{x_1}, \dots,\bb{x_n} \in \RR^d$ be arbitrary, $\bb{x_0}=\0 \in \RR^d$ and\\ 
	$\bb{a_0} = \left(-\bb{1_m^\top}\sum_{k=1}^n \bb{a_k}, 0, \dots,0\right)\in \RR^{m}$. Then
	\begin{align*}
	0 &\le - \sum_{i=0}^n \sum_{j=0}^n   \bb{a_i^\top} \G(\bb{x_i}-\bb{x_j}) \bb{a_j}
	\\   
	&= - \sum_{i=1}^n \sum_{j=1}^n \bb{a_i^\top} \G(\bb{x_i} -\bb{x_j}) \bb{a_j}
	- \sum_{i=1}^n   \bb{a_i^\top}\G(\bb{x_i} -\bb{x_0}) \bb{a_0}
	-  \sum_{j=1}^n   \bb{a_0^\top}  \G(\bb{x_0} -\bb{x_j}) \bb{a_j} \\
	&~~~~- \bb{a_0^\top} \G(\bb{x_0} -\bb{x_0}) \bb{a_0}.
	\end{align*}
	Since $\gamma_{11}(\0)=0$, 
	and $\bb{a_0^\top} \G(\bb{x_0} -\bb{x_j}) \bb{a_j} =\bb{a_j^\top} \G(\bb{x_j}) \bb{a_0}$ due to property \eqref{eq:hermitian},
	we get that
	\begin{align*}
	0 &\le - \sum_{i=0}^n \sum_{j=0}^n   \bb{a_i^\top} \G(\bb{x_i} -\bb{x_j}) \bb{a_j} \\  
	&= \sum_{i=1}^n \sum_{j=1}^n \bb{a_i^\top}\left(\bb{\gamma_1}(\bb{x_i})\bb{1_m^\top} + \bb{1_m\gamma_1^\top}(\bb{x_j})  -	\G(\bb{x_i} -\bb{x_j})\right) \bb{a_j}, 
	\end{align*}
	i.e.\,$(\x,\y) \mapsto \bb{\gamma_1}(\x)\bb{1_m^\top} + \bb{1_m\gamma_1^\top}(\y) -\G(\x -\y)$ is  a matrix-valued positive definite function.
	Let $\{\Z(\x)=(Z_1(\x),\dots, Z_m(\x))^\top, \x \in \RR^d\}$ be a corresponding centred Gaussian random field.
	We have to show that $\Var\left(Z_i(\x+\h) - Z_j(\x)\right)$ is independent of $\x$ for all $\x, \h \in \RR^d$, $i,j=1,\dots,m$.
	We even show that $\x \mapsto Z_i(\x+\h) - Z_j(\x)$ is
	weakly stationary for $i,j=1,\dots,m$:
	\begin{eqnarray*}
		\lefteqn{ \Cov(Z_i(\x+\h) - Z_j(\x), Z_i(\y+\h) - Z_j(\y))}\\
		&=&
		\gamma_{i1}(\x+\h) + \gamma_{i1}(\y+\h) -\gamma_{ii}(\x-\y) +
		\gamma_{j1}(\x) + \gamma_{j1}(\y) - \gamma_{jj}(\x-\y)
		\\&&{}
		-\gamma_{j1}(\x) - \gamma_{i1}(\y+\h) + \gamma_{ji}(\x-\y-\h) -
		\gamma_{i1}(\x+\h) - \gamma_{j1}(\y) +\gamma_{ij}(\x+\h-\y)
		\\&=&
		-\gamma_{ii}(\x-\y) 
		- \gamma_{jj}(\x-\y)
		+ \gamma_{ji}(\x-\y-\h)
		+\gamma_{ij}(\x-\y+\h)
		. \end{eqnarray*}	
\end{proof}
Theorem \ref{thm:characterization1} answers the questions raised in \cite[p.\,422]{du2012} and also settles a question in \cite[p.\,239]{ma2011class} in a more general framework with regard to the intersection of the sets of pseudo- and cross-variograms. It turns out that this intersection is trivial in the following sense.
\begin{corollary} \label{cor:schnitt} Let $\mathcal{P}=\{\G: \RR^d \rightarrow \RR^{m \times m}\mid \G \text{~pseudo-variogram} \}$, $\mathcal{C}=\{\TG: \RR^d \rightarrow \RR^{m \times m}\mid \TG \text{~cross-variogram} \}.$ Then we have 
	\begin{equation} \label{eq:intersection}  \mathcal{P} \cap \mathcal{C} = \{\bb{1_m1_m^\top} \gamma \mid \gamma: \RR^d \rightarrow \RR \text{~variogram} \}. \end{equation}
\end{corollary}
\begin{proof}
	Let $\G \in \mathcal{P} \cap \mathcal{C}$. W.l.o.g.\,assume $m=2$. Since $\G \in \mathcal{P}\cap \mathcal{C}$, we have for $n=2$, $\bb{x_1}=\h$, $\bb{x_2}=\0$, $\bb{a_1},\bb{a_2} \in \RR^2$ with $\bb{1_2^\top} \sum_{k=1}^2 \bb{a_k}=0$, using the symmetry of $\G$ and $\G(\0)=\0$,
	\begin{align*}
	0 &\ge \sum_{i=1}^2\sum_{j=1}^2 \bb{a_i^\top} \G(\bb{x_i}-\bb{x_j})\bb{a_j} \\&= 2a_{11}a_{21} \gamma_{11}(\h)+2a_{12}a_{22}\gamma_{22}(\h)+2(a_{11}a_{22}+a_{12}a_{21})\gamma_{12}(\h) .
	\end{align*}
	Choosing $\bb{a_{1}}=(-1,0)^\top$, $\bb{a_2}= (1-k,k)^\top$, $k \ge 2$, and applying the Cauchy-Schwarz inequality due to $\G \in \mathcal{C}$ gives 
	\begin{align}\label{eq:inequalitycrosspseudo1} 
	0 \le \gamma_{11}(\h) \le \frac{-k}{1-k}\gamma_{12}(\h)  \le \frac{1}{1-\frac1k}\sqrt{\gamma_{11}(\h)}\sqrt{\gamma_{22}(\h)}.
	\end{align}
	By symmetry, we also have 
	\begin{align}\label{eq:inequalitycrosspseudo2}
	0 \le \gamma_{22}(\h) &\le \frac{ -k}{1-k}\gamma_{12}(\h) \le \frac{1}{1-\frac1k}\sqrt{\gamma_{11}(\h)}\sqrt{\gamma_{22}(\h)}.
	\end{align}
	Assume first that, w.l.o.g.\,, $\gamma_{11}(\h)=0$. Then, $\gamma_{12}(\h) = 0$ and $\gamma_{22}(\h) =0$ due to inequalities \eqref{eq:inequalitycrosspseudo1} and \eqref{eq:inequalitycrosspseudo2}. Suppose now that $\gamma_{11}(\h), \gamma_{22}(\h) \neq 0$. Letting $k \rightarrow \infty$ in inequalities \eqref{eq:inequalitycrosspseudo1} and \eqref{eq:inequalitycrosspseudo2} yields $\gamma_{11}(\h) = \gamma_{22}(\h)$.
	Inserting this in inequality \eqref{eq:inequalitycrosspseudo2} gives $$ \gamma_{22}(\h) \le \frac{1}{1-\frac1k}\gamma_{12}(\h) \le \frac{1}{1-\frac1k}\gamma_{22}(\h)$$ and consequently the result for $k \rightarrow \infty$. 
\end{proof}
Corollary \ref{cor:schnitt} can also be proved by means of a result in \cite{oesting2017}. In fact, Theorem \ref{thm:characterization1} enables us to reproduce their result, which was originally derived in a stochastic manner, by a direct proof.

\begin{corollary}\label{cor:oestingsinequality} 
	Let $\G: \RR^d \rightarrow \RR^{m \times m}$ be a pseudo-variogram. Then $\G$ fulfils
	$$\left(\sqrt{\gamma_{ii}(\h)}-\sqrt{\gamma_{ij}(\h)}\right)^2 \le \gamma_{ij}(\0), \quad \h \in \RR^d, \quad i,j=1,\dots,m.$$
\end{corollary}  

\begin{proof}
	W.l.o.g.\,assume $m=2$. We only present the proof for $i=1, j=2$ and $\gamma_{11}(\h)$, $\gamma_{12}(\h) > 0$. We then have for $n=2$, $\bb{x_1}=\h$, $\bb{x_2}=\0$, $\bb{a_1},\bb{a_2} \in \RR^2$ with $\bb{1_2^\top} \sum_{k=1}^2 \bb{a_k}=0$, 
	\begin{eqnarray} \label{eq:aux1}
	0 &\ge& a_{11}a_{21} \gamma_{11}(\h)+a_{12}a_{22}\gamma_{22}(\h) +\nonumber \\ &&a_{11}a_{22}\gamma_{12}(\h) + a_{12}a_{21}\gamma_{21}(\h) + (a_{11}a_{12}+a_{21}a_{22})\gamma_{12}(\0).
	\end{eqnarray}
	Assuming $a_{12}=0$, $a_{22}>0$ and $a_{11} + a_{22} = -a_{21} >0$, inequality \eqref{eq:aux1} simplifies to
	\begin{align} \label{eq:aux2}
	\gamma_{12}(\0) &\ge -\frac{a_{11}}{a_{22}} \gamma_{11}(\h) + \frac{a_{11}}{a_{11}+a_{22}}\gamma_{12}(\h) \nonumber \\
	&= -x \gamma_{11}(\h) + \frac x{1+x}\gamma_{12}(\h)
	\end{align}	
	for $x := \frac{a_{11}}{a_{22}}$. Maximization of the function $x \mapsto -x \gamma_{11}(\h) + \frac x{1+x}\gamma_{12}(\h), x > -1$, leads to $x^{\ast} = \sqrt{\frac{\gamma_{12}(\h)}{\gamma_{11}(\h)}} -1$. Inserting $x^\ast$ into \eqref{eq:aux2} gives 
	\begin{align*}
	\gamma_{12}(\0) &\ge -\left(\sqrt{\frac{\gamma_{12}(\h)}{\gamma_{11}(\h)}}-1\right) \gamma_{11}(\h) + \left( \frac {\sqrt{\frac{\gamma_{12}(\h)}{\gamma_{11}(\h)}} -1}{1+\sqrt{\frac{\gamma_{12}(\h)}{\gamma_{11}(\h)}} -1}     \right)\gamma_{12}(\h) \\ &= \left(\sqrt{\gamma_{11}(\h)}-\sqrt{\gamma_{12}(\h)}\right)^2. 
	\end{align*}
\end{proof}

\section{A Schoenberg-type characterization} \label{sec:schoenberg} 

The proof of Theorem \ref{thm:characterization1} contains an important relation between matrix-valued positive definite and conditionally negative definite functions we have not emphasized yet. Due to its significance, we formulate it in a separate lemma. In fact, the assumption on the main diagonal stemming from our consideration of pseudo-variograms can be dropped, resulting in a matrix-valued counterpart of Lemma 2.1 in \cite{berg1984harmonic}.

\begin{lemma} \label{lem:pdndconstruction}
	Let $\G: \RR^d \rightarrow \RR^{m \times m}$ be a matrix-valued function with $\gamma_{ij}(\h)=\gamma_{ji}(-\h)$, $i,j=1,\dots,m$. Define $$\bb{C_k}(\x,\y):=\bb{\gamma_k}(\x)\bb{1_m^\top} + \bb{1_m}\bb{\gamma_k^\top}(\y) -\G(\x -\y)-\gamma_{kk}(\0)\bb{1_m\mathbf{1}_m^\top}$$ with $\bb{\gamma_k}(\h)=(\gamma_{1k}(\h),\dots,\gamma_{mk}(\h))^\top$, $k \in \{1,\dots,m\}$. Then $\bb{C_k}$ is a positive definite matrix-valued kernel for $k \in \{1,\dots,m\}$, if and only if $\G$ is conditionally negative definite. 
	If $\gamma_{kk}(\0) \ge 0$ for $k=1,\dots,m$, then $$\bb{\tilde{C_k}}(\x,\y):=\bb{\gamma_k}(\x)\bb{1_m^\top} + \bb{1_m}\bb{\gamma_k^\top}(\y) -\G(\x -\y)$$ is a positive definite matrix-valued kernel for $k \in \{1,\dots,m\}$, if and only if $\G$ is conditionally negative definite. 
\end{lemma}

The kernel construction in Lemma \ref{lem:pdndconstruction} enables us to prove an analogue to Schoenberg's result of great significance \cite{berg1984harmonic}. 

\begin{theorem}\label{thm:schoenberg}
	A function $\G: \RR^d \rightarrow \RR^{m \times m}$ is conditionally negative definite, if and only if $\exp^\ast(-t\G)$, with $\exp^\ast(-t\G(\h))_{ij} := \exp(-t\gamma_{ij}(\h))$, is positive definite for all $t >0$.
\end{theorem}

\begin{remark}  Theorem \ref{thm:schoenberg} has also been recently found in \cite{gesztesy2016conditional} in terms of conditionally positive definite matrix-valued functions with complex entries. We give a proof of the result nonetheless, since our proof for the only if part explicitly involves a kernel construction, namely the one from Lemma \ref{lem:pdndconstruction}. Thereby, our kernel provides an "if and only if"-statement in Lemma \ref{lem:pdndconstruction}, whereas the natural multivariate analogue to Schoenberg's kernel which is also considered in \cite{gesztesy2016conditional}, fails to do so, see \cite[Remark 4.10]{gesztesy2016conditional}. We repeat the arguments for the if part, as the construction used there will reappear in the sequel.
\end{remark}

\begin{proof}[Proof of Theorem \ref{thm:schoenberg}]
	Assume that $\G$ is conditionally negative definite. Then, 	
	$$(\x,\y) \mapsto \bb{\gamma_1}(\x)\bb{1_m^\top} + \bb{1_m\gamma_1^\top}(\y)  -
	\G(\x -\y) -\gamma_{11}(\0)\bb{1_m\mathbf{1}_m^\top}
	$$
	is a positive definite kernel due to Lemma~\ref{lem:pdndconstruction}. Since positive definite matrix-valued functions are closed with regard to sums, Hadamard products and pointwise limits, the function
	\begin{align*}(\x,\y) &\mapsto \exp^\ast\left(t\bb{\gamma_1}(\x)\bb{1_m^\top} + t\bb{1_m\gamma_1^\top}(\y)  -
	t\G(\x -\y)-t\gamma_{11}(\0)\bb{1_m\mathbf{1}_m^\top}\right) \\
	&=\exp(-t\gamma_{11}(\0))\exp^\ast\left(t\bb{\gamma_1}(\x)\bb{1_m^\top} + t\bb{1_m\gamma_1^\top}(\y)  -
	t\G(\x -\y)\right)
	\end{align*}
	is again positive definite for all $t >0$. The same holds true for the function
	$(\x,\y) \mapsto  \exp^*\left(-t\bb{\gamma_1}(\x)\bb{1_m^\top} - t\bb{1_m}\bb{\gamma_1^\top}(\y)\right)$ by standard arguments. Using the stability of positive definite functions under Hadamard products again, the first part of the assertion follows. 
	
	Assume now that $\exp^\ast(-t\G)$
	is a positive definite function for all $t>0$. Then, $\exp(-t\gamma_{ij}(\h))=\exp(-t\gamma_{ji}(-\h))$, and thus,
	$$
	\left( \frac{1 - e^{-t\gamma_{ij}}}t\right)_{i,j=1,\ldots,m}
	= \frac{\bb{1_m 1_m^\top} - \exp^\ast(-t\G)}t
	$$
	is a conditionally negative definite function. The assertion follows for $t\rightarrow 0$.
	
\end{proof}

Combining Theorems \ref{thm:characterization1} and \ref{thm:schoenberg}, and recalling that the classes of matrix-valued positive definite functions and covariance functions for multivariate random fields coincide, we immediately get the following characterization of pseudo-variograms.

\begin{corollary} \label{cor:pseudoschoenberg}
	A function $\G: \RR^d \rightarrow \RR^{m \times m}$ is a pseudo-variogram, if and only if $\exp^\ast(-t\G)$ is a matrix-valued correlation function for all $t>0$.
\end{corollary}

Corollary \ref{cor:pseudoschoenberg} establishes a direct link between matrix-valued correlation functions and pseudo-variograms. Together with Corollary \ref{cor:schnitt}, it shows that the cross-variograms for which Theorem 10 in \cite{ma2011vector} holds, are necessarily of the form \eqref{eq:intersection}, and it explains the findings in the first part of Remark 2 in \cite{schlather2010}.

Theorem \ref{thm:schoenberg} can be further generalized for conditionally negative definite functions with non-negative components in terms of componentwise Laplace transforms, providing a matrix-valued version of \cite[Theorem 2.3]{berg1984harmonic}.

\begin{theorem} \label{thm:schoenberglaplace}
	Let $\mu$ be a probability measure on $[0,\infty)$ such that $0 < \int_0^\infty s d\mu(s) < \infty$. Let $\mathcal{L}$ denote its Laplace transform, i.e.\,$\mathcal{L}\mu(x)=\int_0^\infty \exp(-sx) d\mu(s), x \in [0,\infty)$. Then $\G: \RR^d \rightarrow [0,\infty)^{m \times m}$ is conditionally negative definite, if and only if $\left(\mathcal{L}\mu(t\gamma_{ij})\right)_{i,j=1,\dots,m}$ is positive definite for all $t >0$. In particular, $\G$ is a pseudo-variogram, if and only if $\left(\mathcal{L}\mu(t\gamma_{ij})\right)_{i,j=1,\dots,m}$ is an $m$-variate correlation function for all $t >0$.
\end{theorem}

\begin{proof} We follow the proof of the univariate version in \cite{berg1984harmonic}.
	If $\G$ is conditionally negative definite, then $\exp^\ast(-st\G)$ is positive definite for all $s, t >0$ due to Theorem \ref{thm:schoenberg}. Since conditionally negative definite matrix-valued functions are closed under integration, the function  
	$$\left(\mathcal{L}\mu(t\gamma_{ij}(\h))\right)_{i,j=1,\dots,m}= \left(\int_0^\infty \exp(-st\gamma_{ij}(\h)) d\mu(s)\right)_{i,j=1,\dots,m}$$ 
	is positive definite for all $t >0$. Suppose now that $\left(\mathcal{L}\mu(t\gamma_{ij})\right)_{i,j=1,\dots,m}$ is positive definite for all $t >0$. We then have 
	\begin{align*}
	\frac{\bb{1_m1_m^\top} -\left(\mathcal{L}\mu(t\gamma_{ij}(\h))\right)_{i,j=1,\dots,m}}t &=\left(\frac1t (1- \mathcal{L}\mu(t\gamma_{ij}(\h)))\right)_{i,j=1,\dots,m} \\ &= \left(\int_0^\infty \frac{1-\exp(-st\gamma_{ij}(\h))}t d\mu(s)\right)_{i,j=1,\dots,m}.
	\end{align*} Using the dominated convergence theorem due to $\left\lvert\frac{1-\exp(-st\gamma_{ij}(\h))}t\right\rvert \le s\gamma_{ij}(\h)$ for $s,t>0$, $i,j=1,\dots,m$, we get
	\begin{align*}
	\frac{\bb{1_m1_m^\top} -\left(\mathcal{L}\mu(t\gamma_{ij}(\h))\right)_{i,j=1,\dots,m}}t	&\overset{t \rightarrow 0}{\longrightarrow} \left( \gamma_{ij}(\h) \int_0^\infty s d\mu(s)   \right)_{i,j=1,\dots,m}.
	\end{align*} 
	Being a pointwise limit of conditionally negative definite functions, $\G$ itself is conditionally negative definite.
\end{proof}

\begin{corollary} \label{cor:schoenberglaplaceexponential} 
	A function $\G: \RR^d \rightarrow [0,\infty)^{m \times m}$ is a pseudo-variogram, if and only if \begin{equation} \label{eq:schoenberggamma} \left(\left(1 + t\gamma_{ij}(\h)\right)^{-\lambda}\right)_{i,j=1,\dots,m}, \quad \h \in \RR^d, \quad  \lambda >0, \end{equation} is a correlation function of an $m$-variate random field for all $t >0$.
\end{corollary} 
\begin{proof}
	Choose $\mu(ds)=\frac1{\Gamma(\lambda)}\exp(-s)s^{\lambda-1} \ensuremath{\mathbbm{1}}_{(0,\infty)}(s)   ds$ in Theorem \ref{thm:schoenberglaplace}. 
\end{proof}

There are further matrix-valued versions of univariate results. For instance, Bernstein functions also operate on matrix-valued conditionally negative definite functions \cite{berg1984harmonic}, and can thus be used to derive novel pseudo-variograms from known ones. 

\begin{proposition} \label{prop:bernsteinnd}
	Let $\G: \RR^d \rightarrow [0,\infty)^{m\times m}$ be conditionally negative definite. Let $g: [0, \infty) \rightarrow [0,\infty)$ denote the continuous extension of a Bernstein function. Then $\bb{g \circ \G}$ with $\left((\bb{g \circ \G})(\h)\right)_{ij}:= (g\circ \gamma_{ij})(\h)$, $i,j=1,\dots,m$, is conditionally negative definite. In particular, if $g(0)=0$ and $\G$ is a pseudo-variogram, then $\bb{g \circ \gamma}$ is again a pseudo-variogram.
\end{proposition}
\begin{proof}
	Since $g$ is the continuous extension of a Bernstein function, it has a representation $$g(x)=a+bx+\int_0^\infty (1-\exp(-x t))d\nu(t),$$ where $a,b \ge 0$ and $\nu$ is a measure on $[0,\infty)$ with $\int_0^\infty \min(1,t) \nu(dt) < \infty$ \cite{schilling2012}. Thus,
	\begin{align*}
	\left((\bb{g\circ\gamma})(\h)\right)_{ij} &= a+b\gamma_{ij}(\h)+\int_0^\infty (1-\exp(-t\gamma_{ij}(\h)))d\nu(t).
	\end{align*}
	Due to Theorem \ref{thm:schoenberg}, $\bb{g \circ \G}$ is a sum of conditionally negative definite matrix-valued functions and therefore conditionally negative definite itself. 
\end{proof}

\section{Multivariate versions of Gneiting's space-time model} \label{sec:gneitingmultivariat}

Schoenberg's univariate result is often an integral part of proving the validity of covariance models. Here, we use its matrix-valued counterparts derived in the previous section to naturally extend covariance models of Gneiting type to the multivariate case. 

Gneiting's original space-time model is a univariate covariance function on $\RR^d\times \RR$ defined via 
\begin{equation} \label{eq:gneitingoriginal}
G(\h,u)= \frac1{\psi(\lvert u\rvert^2)^{d/2}}\varphi\left(\frac{\lVert \h\rVert^2}{\psi(\lvert u\rvert^2)}\right),  \quad (\h,u) \in \RR^d \times \RR,
\end{equation} 
where $\psi: [0, \infty) \rightarrow [0,\infty)$ is the continuous extension of a Bernstein function, and $\varphi: [0,\infty) \rightarrow [0,\infty)$ is the continuous extension of a bounded completely monotone function \cite{gneiting2002}. For convenience, we simply speak of bounded completely monotone functions henceforth.
Model \eqref{eq:gneitingoriginal} is very popular in practice due to its versatility and capability to model space-time interactions, see \cite{porcu2021review} for a list of several applications. Its special structure has attracted and still attracts interest from a theoretical perspective as well, resulting in several extensions and refinements of the original model~\eqref{eq:gneitingoriginal}, see \cite{zastavnyi2011characterization,menegatto2020,porcu2020gneiting, porcu2016spatio}, 
for instance. Only recently, specific simulation methods for the so-called extended Gneiting class, a special case of \cite[Thereom 2.1]{zastavnyi2011characterization},
\begin{equation}\label{eq:gneitingextended}
G(\h,\uu)= \frac1{(1+\gamma(\uu))^{d/2}}\varphi\left(\frac{\lVert \h\rVert^2}{1+\gamma(\uu)}\right),  \quad (\h,\uu) \in \RR^d \times \RR^l,
\end{equation}
with $\gamma$ denoting a continuous variogram, have been proposed \cite{allard2020simulating}. One of these methods is based on an explicit construction of a random field, where the continuity assumption on $\gamma$ is not needed \cite{allard2020simulating}, and which can be directly transferred to the multivariate case via pseudo-variograms.

\begin{theorem} \label{thm:multivariateextendedgneiting}
	Let $R$ be a non-negative random variable with distribution $\mu$, $\bb{\Omega} \sim N(\0,\mathbf{1}_{d \times d})$ with $\mathbf{1}_{d \times d} \in \RR^{d \times d}$ denoting the identity matrix, $U \sim U(0,1)$, $\Phi\sim U(0,2\pi)$, and let $\bb{W}$ be a centred, $m$-variate Gaussian random field on $\RR^l$ with pseudo-variogram~$\G$, all independent. Then the $m$-variate random field $\Z$ on $\RR^d \times \RR^l$ defined via $$ Z_i(\x,\ttt)=\sqrt{-2\log(U)}\cos\left(\sqrt{2R}\langle\bb{\Omega},\x\rangle + \frac{\lVert \bb{\Omega} \rVert}{\sqrt{2}}W_i(\ttt)+\Phi \right), (\x,\ttt) \in \RR^d\times \RR^l, i =1,\dots,m,$$
	has the extended Gneiting-type covariance function 
	\begin{equation} \label{eq:multivariatextendedGneiting} G_{ij}(\h,\uu)= \frac1{(1+\gamma_{ij}(\uu))^{d/2}}\varphi\left(\frac{\lVert \h\rVert^2}{1+\gamma_{ij}(\uu)}\right), (\h,\uu) \in \RR^d\times \RR^l, \quad i,j=1,\dots,m,
	\end{equation}
	where $\varphi$ denotes a bounded completely monotone function. 
\end{theorem}

\begin{proof}
	The proof follows the lines of the proof of Theorem 3 in \cite{allard2020simulating}. In the multivariate case, the cross-covariance function reads \begin{align*}
	\Cov(Z_i(\x,\ttt),Z_j(\y,\bb{s}))&= \EE \cos\left(\sqrt{2R}\langle\bb{\Omega},\x-\y\rangle + \frac{\lVert \bb{\Omega}\rVert}{\sqrt{2}}\left(W_i(\ttt)-W_j(\bb{s})\right)\right), 
	\end{align*} 
	$(\x,\ttt), (\y,\bb{s}) \in \RR^d \times \RR^l, i,j=1,\dots,m$.
	Due to the assumptions, $W_i(\ttt)-W_j(\bb{s})$ is a Gaussian random variable with mean zero and variance $2\gamma_{ij}(\ttt-\bb{s})$. Further proceeding as in \cite{allard2020simulating} gives the result.
\end{proof}

Theorem \ref{thm:multivariateextendedgneiting} provides a multivariate extension of the extended Gneiting class, and lays the foundations for a simulation algorithm for an approximately Gaussian random field with the respective cross-covariance function, cf.\,\cite{allard2020simulating}. The existence of a Gaussian random field with a preset pseudo-variogram $\G$ and the possibility to sample from it are ensured by Theorem \ref{thm:characterization1} and Lemma \ref{lem:pdndconstruction}, respectively.

Due to our results in previous sections, Theorem \ref{thm:multivariateextendedgneiting} can be easily generalized further, replacing $d/2$ in the denominator in Equation \eqref{eq:multivariatextendedGneiting} by a general parameter $r \ge d/2$. 

\begin{corollary} \label{cor:multivariateextendedGneiting}
	Let $\G: \RR^l \rightarrow \RR^{m \times m}$ be a pseudo-variogram. Then the function $\bb{G}: \RR^d \times \RR^l \rightarrow \RR^{m \times m}$ with 
	\begin{equation} \label{eq:multivariateextendedGneiting2}
	G_{ij}(\h,\uu)= \frac1{(1+\gamma_{ij}(\uu))^{r}}\varphi\left(\frac{\lVert  \h\rVert^2}{1+\gamma_{ij}(\uu)}\right), (\h,\uu) \in \RR^d\times \RR^l, \quad i,j=1,\dots,m,
	\end{equation}
	is a matrix-valued correlation function for $r \ge \frac d2$ and a bounded completely monotone function~$\varphi$.
\end{corollary} 

\begin{proof}
	We already proved the assertion for $r=\frac d2$ in Theorem \ref{thm:multivariateextendedgneiting}. Now, let $\lambda > 0$ and $r = \lambda + \frac d2$. Then the matrix-valued function \eqref{eq:multivariateextendedGneiting2} is the componentwise product of positive definite functions of the form \eqref{eq:schoenberggamma} and \eqref{eq:multivariatextendedGneiting}, and consequently positive definite itself. 
\end{proof}

Even further refinements of Corollary \ref{cor:multivariateextendedGneiting} are possible. We can replace $\bb{1_m1_m^\top} + \G$ in \eqref{eq:multivariateextendedGneiting2} by general conditionally negative definite matrix-valued functions, but for a subclass of completely monotone functions, the so-called generalized Stieltjes functions of order~$\lambda$. This leads to a multivariate version of a result in \cite{menegatto2020}. A bounded generalized Stieltjes function  $S: (0,\infty) \rightarrow [0,\infty)$ of order $\lambda >0$ has a representation
$$S(x) = a + \int_{(0,\infty)} \frac1{(x + v)^\lambda}d\mu(v), \quad x>0,$$ where $a\ge  0$ and the so-called Stieltjes measure $\mu$ is a positive measure on $(0,\infty)$, such that $\int_{(0,\infty)}v^{-\lambda}d\mu(v) < \infty$ \cite{menegatto2020}. As for completely monotone functions, we do not distinguish between a generalized Stieltjes function and its continuous extension in the following. Several examples of generalized Stieltjes functions can be found in \cite{menegatto2020, berg2021}.

\begin{theorem} \label{thm:multivariateGneitingStieltjes}
	Let $S_{ij}, i,j=1,\dots,m$, be generalized Stieltjes functions of order $\lambda >0$. Let the associated Stieltjes measures have densities $\varphi_{ij}$ such that $\left(\varphi_{ij}(v)\right)_{i,j=1,\dots,m}$ is a positive semi-definite matrix for all $v > 0$. Let $\bb{g}: \RR^d \rightarrow [0,\infty)^{m \times m}$, $\bb{f}: \RR^l \rightarrow (0,\infty)^{m \times m}$ be conditionally negative definite functions. Then, the function $\bb{G}: \RR^d\times \RR^l \rightarrow \RR^{m \times m}$ with $$ G_{ij}(\h,\uu)= \frac1{f_{ij}(\uu)^r}S_{ij}\left(\frac{g_{ij}(\h)}{f_{ij}(\uu)}\right), \quad (\h,\uu) \in \RR^d\times\RR^l, \quad i,j=1,\dots,m,   $$ is an $m$-variate covariance function for $r \ge \lambda$.
\end{theorem}

\begin{proof} We follow the proof in \cite{menegatto2020}. It holds that
	\begin{align*} \label{eq:StieltjesGneiting}
	G_{ij}(\h,\uu) 
	&= \frac a{f_{ij}(\uu)^r} + \frac 1{f_{ij}(\uu)^{r-\lambda}} \int_0^\infty \frac1{(g_{ij}(\h) + vf_{ij}(\uu))^\lambda}\varphi_{ij}(v) dv
	\end{align*}
	The function $ x \mapsto \frac1{x^\alpha}$ is completely monotone for $\alpha \ge 0$ and thus the Laplace transform of a measure on $[0,\infty)$ \cite[Theorem 1.4]{schilling2012}. Therefore, $(1/f_{ij}^r)_{i,j=1,\dots,m}$ and $(1/f_{ij}^{r-\lambda})_{i,j=1,\dots,m}$ are positive definite functions due to Theorem \ref{thm:schoenberg} as mixtures of positive definite functions. Furthermore, we have 
	$$ \frac1{(g_{ij}(\h) + vf_{ij}(\uu))^\lambda}= \frac1{\Gamma(\lambda)}\int_0^\infty e^{-sg_{ij}(\h)}e^{-svf_{ij}(\uu)}s^{\lambda -1} ds.$$  The functions $\left(e^{-sg_{ij}(\h)}\right)_{i,j=1,\dots,m}$ and $\left(e^{-svf_{ij}(\uu)}\right)_{i,j=1,\dots,m}$ are again positive definite due to Theorem~\ref{thm:schoenberg} for all $s,v >0$, and so is their componentwise product. Since positive definite functions are closed under integration, $\left(\frac1{(g_{ij}(\h) + vf_{ij}(\uu))^\lambda}\right)_{i,j=1,\dots,m}$ is positive definite for all $v > 0$. Therefore, the function $\left(\frac1{(g_{ij}(\h) + vf_{ij}(\uu))^\lambda}\varphi_{ij}(v)\right)_{i,j=1,\dots,m}$ is also positive definite for all $v > 0$.  Combining and applying the above arguments shows our claim.
\end{proof}

Theorem \ref{thm:multivariateGneitingStieltjes} provides a very flexible model. In a space-time framework, it allows for different covariance structures in both space and time, and it does not require assumptions like continuity of the conditionally negative definite functions involved, or isotropy, which distinguishes it from the multivariate Gneiting-type models presented in \cite{bourotte2016flexible} and \cite{guella2020gaussian}, respectively.

\section*{Acknowledgment}

The authors gratefully acknowledge support by the German Research Foundation (DFG) through the Research Training Group RTG 1953.


\begin{thebibliography}{10}
	\bibitem{allard2020simulating}
	{\sc Allard, D., Emery, X., Lacaux, C. and Lantu{\'e}joul, C.}
	(2020). Simulating space-time random fields with nonseparable Gneiting-type covariance functions. \emph{Stat. Comput.} {\bf 30,} 1479--1495.
	\bibitem{berg1984harmonic}
	{\sc Berg, C., Christensen, J. P. R. and Ressel, P.}
	(1984). \emph{Harmonic Analysis on Semigroups: Theory of Positive Definite and Related Functions},
	Springer, New York.
	\bibitem{berg2021}
	{\sc Berg, C., Koumandos, S. and Pedersen, H. L.}
	(2021). Nielsen's beta function and some infinitely divisible distributions. \emph{Math. Nachr.} {\bf 294,} 426--449.	
	
	\bibitem{bourotte2016flexible}
{\sc Bourotte, M., Allard, D. and Porcu, E.}
(2016). A flexible class of non-separable cross-covariance functions for multivariate space-time data. \emph{Spat. Stat.} {\bf 18,} 125--146.	
	
	

	
	
	\bibitem{chen2019parametric}
	{\sc Chen, W. and Genton, M. G.}
	(2019). Parametric variogram matrices incorporating both bounded and unbounded functions. \emph{Stoch. Environ. Res. Risk Assess.} {\bf 33,} 1669--1679.
	\bibitem{cressie1998variance}
	{\sc Cressie, N. and Wikle, C. K.}
	(1998). The variance-based cross-variogram: you can add apples and oranges. \emph{Math. Geol.} {\bf 30,} 789--799.
	
	
	\bibitem{du2012}
	{\sc Du, J. and Ma, C.}
	(2012). Variogram matrix functions for vector random fields with second-order increments. \emph{Math. Geosci.} {\bf 44,} 411--425.
	\bibitem{genton2015maxstable}
	{\sc Genton, M. G., Padoan, S. A. and Sang, H.}
	(2015). Multivariate max-stable spatial processes. \emph{Biometrika.} {\bf 102,} 215--230.
	\bibitem{gesztesy2016conditional}
	{\sc Gesztesy, F. and Pang, M.}
	(2016). On (conditional) positive semidefiniteness in a matrix-valued context. \emph{arXiv:1602.00384}.
	\bibitem{gneiting2002}
	{\sc Gneiting, T.}
	(2002). Nonseparable, Stationary Covariance Functions for Space-Time Data. \emph{J. Amer. Statist. Assoc.} {\bf 97,} 590--600.
	\bibitem{gneiting2001}
	{\sc Gneiting, T., Sasv{\'a}ri, Z. and Schlather, M.}
	(2001). Analogies and correspondences between variograms and covariance functions. \emph{Adv. Appl. Prob.} {\bf 33,} 617--630.
	
		\bibitem{guella2020gaussian}
	{\sc Guella, J. C.}
	(2020). On Gaussian kernels on Hilbert spaces and kernels on Hyperbolic spaces. \emph{arXiv:2007.14697}.
	
	
	
	
	\bibitem{ma2011class}
	{\sc Ma, C.}
	(2011). A class of variogram matrices for vector random fields in space and/or time. \emph{Math. Geosci.} {\bf 43,} 229--242.
	\bibitem{ma2011vector}
	{\sc Ma, C.}
	(2011). Vector random fields with second-order moments or second-order increments. \emph{Stoch. Anal. Appl.} {\bf 29,} 197--215.
	\bibitem{matheron1972leccon}
	{\sc Matheron, G.}
	(1972). \emph{Le{\c{c}}on sur les fonctions al{\'e}atoire d'ordre 2}, Technical Report C-53, MINES Paristech - Centre de G{\'e}osciences.
	\bibitem{matheron1973intrinsic}
	{\sc Matheron, G.}
	(1973). The intrinsic random functions and their applications. \emph{Adv. Appl. Prob.} {\bf 5,} 439--468.
	\bibitem{menegatto2020}
	{\sc Menegatto, V. A.}
	(2020). Positive definite functions on products of metric spaces via generalized Stieltjes functions. \emph{Proc. Amer. Math. Soc.} {\bf 148,} 4781--4795.
	\bibitem{porcu2020gneiting}
	{\sc Menegatto, V. A., Oliveira, C. P. and Porcu, E.}
	(2020). Gneiting class, semi-metric spaces and isometric embeddings. \emph{Constr. Math. Anal.} {\bf 3,} 85--95.
	\bibitem{myers1982matrix}
	{\sc Myers, D. E.}
	(1982). Matrix formulation of co-kriging. \emph{J. Int. Assoc. Math. Geol.} {\bf 14,} 249--257.
	\bibitem{myers1991pseudo}
	{\sc Myers, D. E.}
	(1991). Pseudo-cross variograms, positive-definiteness, and cokriging. \emph{Math. Geol.} {\bf 23,} 805--816.
	\bibitem{oesting2017}
	{\sc Oesting, M., Schlather, M. and Friederichs, P.}
	(2017). Statistical post-processing of forecasts for extremes using bivariate Brown-Resnick processes with an application to wind gusts. \emph{Extremes.} {\bf 20,} 309--332.
	\bibitem{papritz1993}
	{\sc Papritz, A., K{\"u}nsch, H. R. and Webster, R.}
	(1993). On the pseudo cross-variogram. \emph{Math. Geol.} {\bf 25,} 1015--1026.
	
		\bibitem{porcu2016spatio}
	{\sc Porcu, E., Bevilacqua, M. and Genton, M. G.}
	(2016). Spatio-temporal covariance and cross-covariance functions of the great circle distance on a sphere. \emph{J. Amer. Statist. Assoc.} {\bf 111,} 888--898.
	

	
	
	
	\bibitem{porcu2021review}
{\sc Porcu, E., Furrer, R. and Nychka, D.}
(2021). 30 Years of space--time covariance functions. \emph{Wiley Interdiscip. Rev. Comput. Stat.} {\bf 13,} e1512.
	\bibitem{schilling2012}
	{\sc Schilling, R. L., Song, R. and Vondracek, Z.}
	(2012). \emph{Bernstein functions}, 
	2nd rev. and ext. edn. De Gruyter, Berlin.
	\bibitem{schlather2010}
	{\sc Schlather, M.}
	(2010). Some covariance models based on normal scale mixtures. \emph{Bernoulli.} {\bf 16,} 780--797.
	\bibitem{schoenberg1938metric}
	{\sc Schoenberg, I. J.}
	(1938). Metric spaces and positive definite functions. \emph{Trans. Amer. Math. Soc.} {\bf 44,} 522--536.
	\bibitem{ver1993multivariable}
	{\sc Ver Hoef, J. and Cressie, N.}
	(1993). Multivariable Spatial Prediction. \emph{Math. Geol.} {\bf 25,} 219--240.
	
	\bibitem{wackernagel2003multivariate}
	{\sc Wackernagel, H.}
	(2003). \emph{Multivariate geostatistics: an introduction with applications}, 
	3rd edn. Springer, Berlin.
	
	\bibitem{xie1994positive}
	{\sc Xie, T.}
	(1994). \emph{Positive definite matrix-valued functions and matrix variogram modeling}, PhD thesis, The University of Arizona.
	\bibitem{zastavnyi2011characterization}
	{\sc Zastavnyi, V. P. and Porcu, E.}
	(2011). Characterization theorems for the Gneiting class of space-time covariances. \emph{Bernoulli.} {\bf 17,} 456--465.
	

\end{thebibliography}
\end{document}